\documentclass[12pt]{amsart}

\usepackage{hyperref}
\usepackage{verbatim}
\usepackage{fullpage}
\usepackage{amsfonts}

\usepackage{graphicx}
\usepackage{amsmath, amsthm, amssymb}
\usepackage{fancybox}
\newcommand{\swoosh}{\includegraphics[width=0.35in]{swoosh.pdf}}

\renewcommand{\tilde}{\widetilde}

\newcommand{\R}{\mathbb{R}}

\newcommand{\N}{\mathbb{N}}
\newcommand{\T}{\mathbb{T}}
\newcommand{\Z}{\mathbb{Z}}

\newcommand{\Hil}{\mathcal{H}}

\DeclareMathOperator{\lspan}{span}

\providecommand{\norm}[1]{\lVert#1\rVert}

\newcounter{Theorem}

\numberwithin{equation}{section}
\numberwithin{Theorem}{section}

\theoremstyle{plain} 
\newtheorem{thm}[Theorem]{Theorem}
\newtheorem{cor}[Theorem]{Corollary}
\newtheorem{lem}[Theorem]{Lemma}

\theoremstyle{definition}
\newtheorem{defn}[Theorem]{Definition}

\theoremstyle{remark}

\newtheorem{ex}[Theorem]{Example}

\newtheorem{problem}{Problem}

\begin{document}
\title{Constructive proof of the Carpenter's Theorem}

\author{Marcin Bownik}
\author{John Jasper}

\address{Department of Mathematics, University of Oregon, Eugene, OR 97403--1222, USA}

\email{mbownik@uoregon.edu}

\address{Department of Mathematics, University of Missouri, Columbia, MO 65211--4100, USA}

\email{jasperj@missouri.edu}

\keywords{diagonals of projections, the Schur-Horn theorem, the Pythagorean theorem, the Carpenter theorem, spectral theory}

\keywords{diagonals of self-adjoint operators, the Schur-Horn theorem, the Pythagorean theorem, the Carpenter theorem, spectral theory}

\thanks{
This work was partially supported by a grant from the Simons Foundation (\#244422 to Marcin Bownik).
The second author was supported by NSF ATD 1042701}

\subjclass[2000]{Primary: 42C15, 47B15, Secondary: 46C05}
\date{\today}

\begin{abstract}
We give a constructive proof of Carpenter's Theorem due to Kadison \cite{k1,k2}. Unlike the original proof our approach also yields the real case of this theorem.
\end{abstract}

\maketitle

\section{Kadison's theorem}

In \cite{k1} and \cite{k2} Kadison gave a complete characterization of the diagonals of orthogonal projections on a Hilbert space $\mathcal H$. 

\begin{thm}[Kadison]\label{Kadison} Let $\{d_{i}\}_{i\in I}$ be a sequence in $[0,1]$. Define
\[a=\sum_{d_{i}<1/2}d_{i} \quad\text{and}\quad b=\sum_{d_{i}\geq 1/2}(1-d_{i}).\]
There exists a projection $P$ with diagonal $\{d_{i}\}$ if and only if one of the following holds
\begin{enumerate}
\item $a,b<\infty$ and  $a-b\in\Z$,
\item $a=\infty$ or $b=\infty$.
\end{enumerate}
\end{thm}

The goal of this paper is to give a constructive proof of the sufficiency direction of Kadison's theorem. Kadison \cite{k1,k2} referred to the necessity part of Theorem \ref{Kadison} as the Pythagorean Theorem and the sufficiency as Carpenter's Theorem. Arveson \cite{a} gave a necessary condition on the diagonals of a certain class of normal operators with finite spectrum. When specialized to the case of two point spectrum Arveson's theorem yields the Pythagorean Theorem, i.e., the necessity of (i) or (ii) in Theorem \ref{Kadison}. Whereas Kadison's original proof is a beautiful direct argument, Arveson's proof uses the Fredholm Index Theory.

In contrast, until very recently there were no proofs of Carpenter's Theorem other than the original one by Kadison, although its extension for $\rm{II}_1$ factors was studied by Argerami and Massey \cite{am}. A notable exception is a recent paper by Argerami \cite{ar} about which we became aware only after completing this work.  In this paper we give an alternative proof of Carpenter's Theorem which has two main advantages over the original. First, the original proof does not yield the real case, which ours does. Second, our proof is constructive in the sense that it gives a concrete algorithmic process for finding the desired projection. This is distinct from Kadison's original proof, which is mostly existential.

The paper is organized as follows. In Section 2 we state preliminary results such as finite rank Horn's theorem. These results are then used in Section 3 to show the sufficiency of (i) in Theorem \ref{Kadison}. The key role in the proof is played by a lemma from \cite{mbjj} which enables modifications of diagonal sequences into more favorable configurations. Section 4 contains the proof of sufficiency of (ii) in Theorem \ref{Kadison}. To this end we introduce an algorithmic procedure for constructing a projection with prescribed diagonal which is reminiscent of the spectral tetris construction introduced by Casazza et al. \cite{cfmwz} in their study of tight fusion frames. Finally, in Section 5 we formulate an open problem of characterizing spectral functions of shift-invariant spaces in $L^2(\mathbb R^d)$, introduced by the first author and Rzeszotnik in \cite{br}, which was a motivating force behind this paper.

\section{Preliminary results}

The main goal of this section is to give a constructive proof of  Horn's Theorem \cite[Theorem 9.B.2]{moa}, which is the sufficiency part of the Schur-Horn Theorem \cite{horn, schur}. We present this proof both for the sake of self-sufficiency of part (i) of Carpenter's Theorem and also to cover the more general case of finite rank operators on an infinite dimensional Hilbert space, see also \cite{ak, kw0, kw}. Moreover, we also give an argument reducing Theorem \ref{Kadison} to the countable case.

\begin{thm}[Horn's Theorem]\label{Horn finite rank} Let $\{\lambda_{i}\}_{i=1}^{N}$ be a positive nonincreasing sequence, and let $\{d_{i}\}_{i=1}^{M}$ be a nonnegative nonincreasing sequence, where $M\in\N \cup\{\infty\}$ and $M\ge N$. If 
\begin{equation}\label{finite rank majorization}\begin{split}
\sum_{i=1}^{n}d_{i} \leq \sum_{i=1}^{n}\lambda_{i} & \quad \text{for all }n\leq N,\\
\sum_{i=1}^{M}d_{i} = \sum_{i=1}^{N}\lambda_{i}, & \\
\end{split}\end{equation}
then there is a positive rank $N$ operator $S$ on a real $M$-dimensional Hilbert space $\Hil$ with positive eigenvalues $\{\lambda_{i}\}_{i=1}^{N}$ and diagonal $\{d_{i}\}_{i=1}^{M}$.
\end{thm}

We need a basic lemma.

\begin{lem}\label{Horn rank 1} Let $M\in\N\cup\{\infty\}$. If $\{d_{i}\}_{i=1}^{M}$ is a nonzero nonnegative sequence with
\[\sum_{i=1}^{M}d_{i}=\lambda<\infty,\]
then there is a positive rank $1$ operator $S$ on an $M$-dimensional Hilbert space $\mathcal H$ with eigenvalue $\lambda$ and diagonal $\{d_{i}\}$.\end{lem}

\begin{proof} Let $\{e_{i}\}_{i=1}^{M}$ be an orthonormal basis for the Hilbert space $\Hil$. Set
\[v = \sum_{i=1}^{M}\sqrt{d_{i}}e_{i},\]
and define $S:\Hil\to\Hil$ by $Sf = \langle f,v\rangle v$ for each $f\in\Hil$. Clearly $S$ is rank $1$, and since $\norm{v}^{2}=\lambda$ the vector $v$ is an eigenvector with eigenvalue $\lambda$. Finally, it is simple to check that $S$ has the desired diagonal.\end{proof}

\begin{proof}[Proof of Theorem \ref{Horn finite rank}] The proof proceeds by induction on $N$. The base case $N=1$ follows from Lemma \ref{Horn rank 1}. Suppose that Theorem \ref{Horn finite rank} holds for ranks up to $N-1$.
Define
\[m_{0}=\max\bigg\{m :\sum_{i=m}^{M}d_{i}\geq \lambda_{N}\bigg\}\]
and
\begin{equation}\label{eta}\eta=\bigg(\sum_{i=m_{0}}^{M}d_{i}\bigg)-\lambda_{N} = \sum_{i=1}^{N-1}\lambda_{i} - \sum_{i=1}^{m_{0}-1}d_{i}.\end{equation}
Note that the maximality of $m_{0}$ implies that $m_{0}\geq N$. For each $n\leq N$ define
\[\delta_{n} = \sum_{i=1}^{n}(\lambda_{i}-d_{i}) \ge 0.\]
For a certain value $0 \le \Delta\leq\eta$, which will be specified later, define the sequence
\begin{equation}\label{td}\tilde{d}_{i} = \begin{cases} d_{1}+\Delta & i=1\\ d_{m_{0}}-\Delta & i=m_{0}\\ d_{i} & i\neq 1,m_{0}.\end{cases}\end{equation}
From the maximality of $m_{0}$ we have
\[\tilde{d}_{m_{0}} = d_{m_{0}}-\Delta\geq d_{m_{0}}-\eta = \lambda_{N}-\sum_{i=m_{0}+1}^{M}d_{i}>0.\]
This shows that $\{\tilde{d}_{i}\}$ is a nonnegative sequence. However, note that this sequence might might fail to be nonincreasing at the position $i=m_0$, which requires extra care in our considerations.

Our next goal is to construct an operator $\tilde{S}$ with positive eigenvalues $\{\lambda_{i}\}_{i=1}^{N}$, diagonal $\{\tilde{d}_{i}\}_{i=1}^{M}$ with respect to the orthonormal basis $\{e_{i}\}_{i=1}^{M}$, and the property that $\langle\tilde{S}e_{1},e_{m_{0}}\rangle = 0$.
The argument splits into two cases.

\textbf{Case 1:} Assume there exists $n\leq \min\{N,m_{0}-1\}$ such that $\delta_{n}<\eta$. Fix $n_{0}\leq \min\{N,m_{0}-1\}$ such that $\delta_{n_{0}}\leq \delta_{n}$ for all $n\leq \min\{N,m_{0}-1\}$. Define $\{\tilde{d}_{i}\}$ as in \eqref{td} with $\Delta = \delta_{n_{0}}$.

Note that
\begin{equation}\label{horn1}\sum_{i=n_{0}+1}^{M}\tilde{d}_{i} = -\delta_{n_{0}} + \sum_{i=n_{0}+1}^{M}d_{i} = \sum_{i=1}^{M}d_{i} - \sum_{i=1}^{n_{0}}\lambda_{i} = \sum_{i=1}^{N}\lambda_{i} - \sum_{i=1}^{n_{0}}\lambda_{i} = \sum_{i=n_{0}+1}^{N}\lambda_{i}.\end{equation}
Since $m_{0}>n_{0}$ and $\tilde{d}_{m_{0}}>0$, from \eqref{horn1} we see that $n_{0}<N$.

For $n\leq n_{0}$
\[\sum_{i=1}^{n}\tilde{d}_{i} = \delta_{n_{0}} + \sum_{i=1}^{n}d_{i} \leq \delta_{n}+\sum_{i=1}^{n}d_{i} = \sum_{i=1}^{n}\lambda_{i}\]
with equality when $n=n_{0}$. Since $n_{0}<N$, by the inductive hypothesis there is a positive rank $n_{0}$ operator $\tilde{S}_{1}$ with eigenvalues $\{\lambda_{i}\}_{i=1}^{n_{0}}$ and diagonal $\{\tilde{d}_{i}\}_{i=1}^{n_{0}}$ with respect to the basis $\{e_{i}\}_{i=1}^{n_{0}}$.

Observe that the subsequence $\{\tilde d_i\}_{i=n_0+1}^{N-1}$  coincides with $\{ d_i\}_{i=n_0+1}^{N-1}$ since $N-1<m_0$. Thus, for any $n_0+1\le n  \le N-1$ we have
\[\sum_{i=n_{0}+1}^{n} \tilde d_{i} = \sum_{i=n_{0}+1}^{n}d_{i} \leq \delta_{n}-\delta_{n_{0}} + \sum_{i=n_{0}+1}^{n}d_{i} = \sum_{i=n_{0}+1}^{n}\lambda_{i}.\]
Moreover, by \eqref{horn1} we have
\[
\sum_{i=n_{0}+1}^{N}\tilde d_{i} \leq \sum_{i=n_{0}+1}^{M}\tilde{d}_{i} = \sum_{i=n_{0}+1}^{N}\lambda_{i}.\]
Thus, $\{\lambda_{i}\}_{i=n_{0}+1}^{N}$ and the nonincreasing rearrangement of $\{\tilde d_i\}_{i=n_0+1}^M$ satisfy the inductive hypothesis \eqref{finite rank majorization}. That is, there is a positive rank $N-n_{0}$ operator $\tilde{S}_{2}$ with  eigenvalues $\{\lambda_{i}\}_{i=n_{0}+1}^{N}$ and diagonal $\{\tilde{d}_{i}\}_{i=n_{0}+1}^{M}$ with respect to the basis $\{e_{i}\}_{i=n_{0}+1}^M$. Thus, the operator $\tilde{S} = \tilde{S}_{1}\oplus\tilde{S}_{2}$ has the desired properties. Indeed, the property that $\langle\tilde{S}e_{1},e_{m_{0}}\rangle = 0$ follows immediately from the definition of $\tilde S$ and the fact that $n_0<m_0$.

\textbf{Case 2:} Assume $\eta\leq \delta_{n}$ for all $n\leq \min\{N,m_{0}-1\}$. Define $\{\tilde{d}_{i}\}$ as in \eqref{td} with $\Delta = \eta$. For $n\leq N-1$ we have
\[\sum_{i=1}^{n}\tilde{d}_{i} = \eta+\sum_{i=1}^{n}d_{i}\leq \delta_{n}+\sum_{i=1}^{n}d_{i} = \sum_{i=1}^{n}\lambda_{i}.\]
We also have by \eqref{eta}
\[\sum_{i=1}^{m_{0}-1}\tilde{d}_{i} = \eta + \sum_{i=1}^{m_{0}-1}d_{i} = \sum_{i=1}^{N-1}\lambda_{i}.\]
By the inductive hypothesis there is a positive rank $N-1$ operator $\tilde{S}_{1}$ with diagonal $\{\tilde{d}_{i}\}_{i=1}^{m_{0}-1}$ and positive eigenvalues $\{\lambda_{i}\}_{i=1}^{N-1}$.
Using the equality in \eqref{finite rank majorization} we have
\[\sum_{i=m_{0}}^{M}\tilde{d}_{i} = -\eta + \sum_{i=m_{0}}^{M}d_{i} = \sum_{i=1}^{M}d_{i} - \sum_{i=1}^{N-1}\lambda_{i} = \lambda_{N}.\]
By Lemma \ref{Horn rank 1} there is a positive rank $1$ operator $\tilde{S}_{2}$ with diagonal $\{\tilde{d}_{i}\}_{i=m_{0}}^{M}$ and eigenvalue $\lambda_{N}$. Thus, the operator $\tilde{S} = \tilde{S}_{1}\oplus\tilde{S}_{2}$ has the desired properties.

Combining the above two cases shows that the desired operator $\tilde{S}$ exists. Let $\alpha\in[0,1]$ be such that $\alpha(d_{1}+\Delta) + (1-\alpha)(d_{m_{0}}-\Delta)=d_{1}$. Define the unitary operator $U$ on the orthonormal basis $\{e_{i}\}_{i=1}^{M}$ by
\begin{equation*}
U(e_{i})=\begin{cases}
\sqrt{\alpha}e_{1} - \sqrt{1-\alpha}e_{m_{0}} & i=1,\\
\sqrt{1-\alpha}e_{1} + \sqrt{\alpha}e_{m_{0}} & i=m_{0},\\
e_{i} & \text{otherwise}.
\end{cases}
\end{equation*}
A simple calculation shows that $S=U^{\ast}\tilde{S}U$ has diagonal $\{d_{i}\}_{i=1}^{M}$ in the basis $\{e_{i}\}_{i=1}^{M}$. This completes the proof of Theorem \ref{Horn finite rank}.

\end{proof}

The following ``moving toward $0$-$1$'' lemma first appeared in \cite{mbjj}. Its proof is constructive as it consists a finite number of ``convex moves'' as at the end of the previous proof. Moreover, from the proof in \cite{mbjj} it follows that Lemma \ref{ops} holds for real Hilbert spaces as well as complex.

\begin{lem}\label{ops}
 Let $\{d_{i}\}_{i\in I}$ be a sequence in $[0,1]$. Let $I_0, I_1 \subset I$ be two disjoint finite subsets such that $\max\{d_{i}: i \in I_0\}\leq\min\{d_{i}: i \in I_1\}$. Let $\eta_{0}\geq 0$ and
\[
\eta_{0}\leq\min\bigg\{\sum_{i\in I_0} d_{i},\sum_{i\in I_1}
(1-d_{i}) \bigg\}.\]
(i) There exists a sequence $\{\tilde d_{i}\}_{i\in I}$  in $[0,1]$
satisfying
\begin{align}
\label{ops0}
\tilde d_i = d_i &\quad\text{for } i \in I \setminus (I_0 \cup I_1),
\\
\label{ops1}
\tilde d_i \leq d_i \quad i\in I_0,
&\quad\text{and}\quad
\tilde d_i \ge d_i, \quad i\in I_1,
\\
\label{ops2}
\eta_0+\sum_{i\in I_0}\tilde{d}_{i} =\sum_{i\in I_0}d_{i}
&\quad\text{and}\quad 
\eta_0+\sum_{i\in I_1} (1-\tilde{d}_{i})=\sum_{i\in I_1} (1-d_{i}).
\end{align}
(ii) For any self-adjoint operator $\tilde E$ on $\Hil$ with diagonal $\{\tilde d_{i}\}_{i\in I}$,
there exists an operator $E$ on $\Hil$ unitarily equivalent to $\tilde E$ with diagonal $\{d_{i}\}_{i\in I}$.
\end{lem}

We end this section by remarking that the indexing set $I$ in Theorem \ref{Kadison} need not be countable. In \cite{k2} the possibility that $I$ is an uncountable set is addressed in all but the most difficult case where $\{d_{i}\}$ and $\{1-d_{i}\}$ are nonsummable \cite[Theorem 15]{k2}. However, the case when $I$ is uncountable is a simple extension of the countable case, as we explain below.

\begin{proof}[Proof of reduction of Theorem \ref{Kadison} to countable case]
First, we consider a projection $P$ with diagonal $\{d_{i}\}_{i\in I}$ with respect to some orthonormal basis $\{e_{i}\}_{i\in I}$ of a Hilbert space $\mathcal H$. If $a$ or $b$ is infinite then there is nothing to show, so we may assume $a,b<\infty$. Set $J = \{i\in I:d_{i}=0\}\cup\{i\in I:d_{i} = 1\}$, and let $P'$ be the restriction of $P$ to the subspace $\mathcal H'=\overline{\lspan}\{e_{i}\}_{i\in I\setminus J}$. Since $e_{i}$ is an eigenvector for each $i\in J$, $\mathcal H'$ is an invariant subspace $P'(\mathcal H') \subset \mathcal H'$. Hence, $P'$ is a projection with diagonal $\{d_{i}\}_{i\in I\setminus J}$. The assumption that $a,b<\infty$ implies $I\setminus J$ is at most countable. Thus, the countable case of Theorem \ref{Kadison} applied to the operator $P'$ yields $a-b\in\Z$. This shows that (ii) is necessary. 

To show that (i) or (ii) is sufficient, we claim that it is enough to assume that all of $d_{i}$'s are in $(0,1)$. If we can find a projection $P$ with only these $d_{i}$'s, then we take $\mathbf I$ to be the identity and ${\bf 0}$ the zero operator on Hilbert spaces with dimensions equal to the cardinalities of the sets $\{i\in I: d_i=1\}$ and $\{i\in I: d_i=0\}$, respectively. Then, $P\oplus \mathbf I\oplus {\bf 0}$ has diagonal $\{d_{i}\}$. Since $a$ and $b$ do not change when we restrict to $(0,1)$, we may assume that $\{d_{i}\}_{i\in I}$ has uncountably many terms and is contained in $(0,1)$. There is some $n\in\N$ such that $J=\{i\in I:1/n < d_{i}<1-1/n\}$ has the same cardinality as $I$. Thus, we can partition $I$ into a collection of countable infinite sets $\{I_{k}\}_{k\in K}$ such that $I_{k}\cap J$ is infinite for each $k\in K$. Each sequence $\{d_{i}\}_{i\in I_{k}}$ contains infinitely many terms bounded away from $0$ and $1$, thus (ii) holds. Again, by the countable case of Theorem \ref{Kadison}, for each $k\in K$ there is a projection $P_{k}$ with diagonal $\{d_{i}\}_{i\in I_{k}}$. Thus, $\bigoplus_{k\in K}P_{k}$ is a projection with diagonal $\{d_{i}\}_{i\in I}$.
\end{proof}

\section{Carpenter's Theorem part i}

The goal of this section is to give a proof of the sufficiency of (i) in Theorem \ref{Kadison}. As a corollary of Theorem \ref{Horn finite rank} we have the summable version of the Carpenter's Theorem.

\begin{thm}\label{fcpt} Let $M\in\N\cup\{\infty\}$, and let $\{d_{i}\}_{i=1}^{M}$ be a sequence in $[0,1]$. If $\sum_{i=1}^{M}d_{i}\in\N$, then there is a projection $P$ with diagonal $\{d_{i}\}$.\end{thm}

\begin{proof} Let $\{d_{i}'\}_{i=1}^{M'}$ be the terms of $\{d_{i}\}$ in $(0,1]$, listed in nonincreasing order. Set $N=\sum_{i=1}^{M}d_{i}$, and define $\lambda_{i}=1$ for $i=1,\ldots,N$. Since $d_{i}'\leq 1$ for all $i$ we have
\begin{equation}\label{frcpt1}\sum_{i=1}^{n}d_{i}'\leq \sum_{i=1}^{n}\lambda_{i}\qquad \text{for } n=1,2,\ldots,N.\end{equation}
We also have 
\[\sum_{i=1}^{M'}d_{i}' = N = \sum_{i=1}^{N}\lambda_{i}.\]
By Theorem \ref{Horn finite rank} there is a rank $N$ self-adjoint operator $P'$ with positive eigenvalues $\{\lambda_{i}\}_{i=1}^{N}$ and diagonal $\{d_{i}'\}_{i=1}^{M'}$. Since $\lambda_{i} = 1$ for each $i$, the operator $P'$ is a projection. Let $\mathbf{0}$ be the zero operator on a Hilbert space with dimension equal to $|\{i\colon d_{i}=0\}|$. The operator $P'\oplus \mathbf{0}$ is a projection with diagonal $\{d_{i}\}_{i=1}^{M}$.
\end{proof}

\begin{cor}\label{cptf} Let $M\in\N\cup\{\infty\}$ and $\{d_{i}\}_{i=1}^{M}$ be a sequence in $[0,1]$. If $\sum_{i=1}^{M}(1-d_{i})\in\N$, then there is a projection $P$ with diagonal $\{d_{i}\}$.
\end{cor}

\begin{proof}
This follows immediately from the observation that a projection $P$ has diagonal $\{d_{i}\}$ if and only if $\mathbf I-P$ is a projection with diagonal $\{1-d_{i}\}$.
\end{proof}

Finally, we can handle the general case (i) of the Carpenter's Theorem.

\begin{thm}\label{cptk} Let $\{d_{i}\}_{i\in I}$ be a sequence in $[0,1]$. If 
\begin{equation}\label{cptk1}
a = \sum_{d_{i}<1/2}d_{i}<\infty,\  b =\sum_{d_{i}\geq 1/2}(1-d_{i}) < \infty,
 \text{ and } a-b\in\Z,\end{equation}
then there exists a projection $P$ with diagonal $\{d_{i}\}$.
\end{thm}

\begin{proof} First, note that if $\{d_{i}\}$ or $\{1-d_{i}\}$ is summable, then by \eqref{cptk1} its sum is in $\N$. Thus, we can appeal to Theorem \ref{fcpt} or Corollary \ref{cptf}, resp., to obtain the desired projection. Hence, we may assume both $0$ and $1$ are limit points of the sequence $\{d_{i}\}$.

Next, we claim that it is enough to prove the theorem under the assumption that $d_{i}\in(0,1)$ for all $i$. Indeed, if $P$ is a projection with diagonal $\{d_{i}\}_{d_{i}\in(0,1)}$, $\mathbf I$ is the identity operator on a space of dimension $|\{i\colon d_{i}=1\}|$, and $\mathbf{0}$ is the zero operator on a space of dimension $|\{i\colon d_{i}=0\}|$, then $P\oplus \mathbf I\oplus\mathbf{0}$ is a projection with diagonal $\{d_{i}\}_{i\in I}$.

Define $J_{0} = \{i\in I:d_{i}<1/2\}$ and $J_{1}=\{i\in I:d_{i}\geq1/2\}$. Choose $i_{1}\in J_{1}$ such that $d_{i_{1}}\leq d_{i}$ for all $i\in J_{1}$. Choose $J'_{0}\subseteq J_{0}$ such that $J_{0}\setminus J'_{0}$ is finite and 
\[
\sum_{i\in J'_{0}}d_{i}<1-d_{i_{1}}.
\]
Let $i_{2}\in J_{1}$ be such that $d_{i_{2}}> d_{i_{1}}$ and
\[d_{i_{2}} + \sum_{i\in J'_{0}}d_{i} \geq 1.\]
Set
\begin{equation}\label{cptk2}
\eta_{0} = \sum_{i\in J'_{0}}d_{i}-(1-d_{i_{2}} ) <\sum_{i\in J'_{0}}d_{i}< 1- d_{i_1}.
\end{equation}
Let $I_0\subset J'_0$ be a finite set such that
\begin{equation}\label{cptk3}
\sum_{i\in I_0}d_{i}>\eta_{0}.
\end{equation}
By \eqref{cptk2} and \eqref{cptk3}, we can apply Lemma \ref{ops} to finite subsets $I_0$ and $I_1=\{i_1\}$ to obtain a sequence $\{\tilde{d}_{i}\}_{i\in I}$ coinciding with $\{d_{i}\}_{i\in I}$ outside of $I_0\cup I_1$ and such that
\[
\sum_{i\in I_0}\tilde{d}_{i} = \sum_{i\in I_0}d_{i} - \eta_{0}\qquad\text{and}\qquad 1-\tilde{d}_{i_{1}} = 1-d_{i_{1}} - \eta_{0}.\]
Note that
\[
\sum_{i\in J'_0\cup\{i_{2}\}}\tilde{d}_{i} = d_{i_{2}} + \sum_{i\in J'_{0}\setminus I_{0}}d_{i} + \sum_{i\in I_{0}}\tilde{d}_{i} = d_{i_{2}} + \sum_{i\in J'_{0}\setminus I_{0}}d_{i} + \sum_{i\in I_{0}}d_{i} - \eta_{0} = 1.
\]
Thus, by Theorem \ref{fcpt} there is a projection $P_{1}$ with diagonal $\{\tilde{d}_{i}\}_{i\in J'_{0}\cup\{i_{2}\}}$.
Next, we note that
\begin{align*}
\sum_{i\in I\setminus(J'_{0}\cup\{i_{2}\})} (1-\tilde{d}_{i})
&=
\sum_{i\in J_{0}\setminus J'_{0}}(1-\tilde{d}_{i}) + \sum_{i\in J_{1}\setminus \{i_{2}\}}(1-\tilde{d}_{i}) 
\\
&
= |J_{0}\setminus J'_{0}| - \sum_{i\in J_{0}\setminus J'_{0}}d_{i} + \sum_{i\in J_{1}\setminus \{i_{2}\}}(1-d_{i}) - \eta_{0}\\
 & = |J_{0}\setminus J'_{0}| - \sum_{i\in J_{0}}d_{i} + \sum_{i\in J_{1}}(1-d_{i}) = |J_{0}\setminus J'_{0}| - a+b\in\N.
\end{align*}
By Corollary \ref{cptf} there is a projection $P_{2}$ with diagonal $\{\tilde{d}_{i}\}_{i\in I\setminus(J'_{0}\cup\{i_{2}\})}$.

The projection $P_{1}\oplus P_{2}$ has diagonal $\{\tilde{d}_{i}\}_{i\in I}$. By Lemma \ref{ops} (ii) there is an operator $P$ with diagonal $\{d_{i}\}_{i\in I}$ which is unitarily equivalent to $P_{1}\oplus P_{2}$. Thus, $P$ is the required projection.
\end{proof}

In \cite[Remark 8]{k1} Kadison asked whether it is possible to construct projections with specified diagonal so that all its entries are real and nonnegative. While the answer is positive for rank one, in general it is negative for higher rank projections.

\begin{ex}
Consider any sequence $\{d_i\}_{i=1}^3$ of numbers in $(0,1)$ such that $d_1+d_2+d_3=2$. By Theorem \ref{fcpt} there exists a projection $P$ on $\R^3$ with such diagonal. However, some entries of $P$ must be negative. Indeed, $\mathbf I -P$ is rank one projection. Hence, $(\mathbf I -P)x=\langle x, v \rangle v $ for some unit vector $v=(v_1,v_2,v_3)\in \R^3$. That is, $(i,j)$ entry of $\mathbf I -P$ equals $v_iv_j$. In particular, $(v_i)^2=1-d_i>0$ for each $i$. This implies that for some $i\ne j$, the off-diagonal entry $(i,j)$ of $\mathbf I -P$ must be positive. Consequently, $(i,j)$ entry of $P$ is negative.
\end{ex}

\section{The algorithm and Carpenter's Theorem part ii}

In this section we introduce an algorithmic technique for finding a projection with prescribed diagonal. The main result of this section is Theorem \ref{cptalg}. Given a non-summable sequence $\{d_{i}\}$ with all terms in $[0,1/2]$, except possibly one term in $(1/2,1)$, Theorem \ref{cptalg} produces an  orthogonal projection with the diagonal $\{d_{i}\}$. Applying this result countably many times allows us to deal with all possible diagonal sequences in part (ii) of Carpenter's Theorem.

The procedure of Theorem \ref{cptalg} is reminiscent to spectral tetris construction of tight frames introduced by Casazza et al. in \cite{cfmwz}, and further investigated in \cite{chkwz}. In fact, the infinite matrix constructed in the proof of Theorem \ref{cptalg} consists of column vectors forming a Parseval frame with squared norms prescribed by the sequence $\{d_{i}\}$. However, our construction was discovered independently with a totally different aim than that of \cite{cfmwz}.

\begin{lem}\label{translem} Let $\sigma,d_{1},d_{2}\in[0,1]$. If $\max\{d_{1},d_{2}\}\leq\sigma$ and $\sigma\leq d_{1}+d_{2}$, then there exists a number $a\in[0,1]$ such that the matrix
\begin{equation}\label{transmatrix}
\left[\begin{array}{cc}
a & \sigma-a\\
d_{1}-a & d_{2}-\sigma+a
\end{array}\right]
\end{equation}
has entries in $[0,1]$ and
\begin{equation}\label{transeq}a(d_{1}-a)=(\sigma-a)(d_{2}-\sigma+a).\end{equation}
Moreover, if $d_{1}+d_{2}<2\sigma$, then $a$ is unique and given by
\begin{equation}\label{transsol} a=\frac{\sigma(\sigma-d_{2})}{2\sigma-d_{1}-d_{2}}.\end{equation}
\end{lem}

\begin{proof} First, assume $\max\{d_{1},d_{2}\}\leq\sigma$ and $\sigma\leq d_{1}+d_{2}$. If $d_{1}=d_{2}=\sigma$ then any $a\in[0,\sigma]$ will satisfy \eqref{transeq} and the matrix \eqref{transmatrix} will have entries in [0,1]. Thus, we may additionally assume $d_{1}+d_{2}<2\sigma$, and hence $\sigma>0$. Since the quadratic terms in \eqref{transeq} cancel out, the equation is linear and the unique solution is given by \eqref{transsol}. It remains to show that the entries of the matrix in \eqref{transmatrix} are in $[0,1]$.
It is clear that $a\geq 0$. Next, we calculate
\begin{equation}\label{translem.1}\sigma-a = \sigma\left(1-\frac{\sigma-d_{2}}{2\sigma-d_{1}-d_{2}}\right) = \frac{\sigma(\sigma-d_{1})}{2\sigma-d_{1}-d_{2}},\end{equation}
which implies that $\sigma-a\geq 0$. Since $\sigma\leq1$ we clearly have $a,\sigma-a\in[0,1]$. It remains to prove that the second row of \eqref{transmatrix} has nonnegative entries. Since $d_{1}+d_{2}\in[\sigma,2\sigma)$ we have
\[(d_{1}-a) + (d_{2}-\sigma+a) = d_{1}+d_{2} - \sigma\in[0,\sigma).\]
If one of $d_{1}-a$ or $d_{2}-\sigma+a$ is negative, then the other must be positive. From \eqref{transeq} we see that $a=\sigma-a=0$. This contradicts the assumption that $\sigma > 0$. Thus, both $d_{1}-a$ and $d_{2}-\sigma+a$ are nonnegative.
\end{proof}

\begin{lem}\label{inj} Let $\{d_{i}\}_{i\in \N}$ be a sequence such that $d_{1}\in [0,1)$, $d_{i}\in[0,\frac{1}{2}]$ for $i\geq 2$ and $\sum_{i=1}^{\infty}d_{i}=\infty$. There is a bijection $\pi:\N\to\N$ such that for each $n\in\N$ we have
\begin{equation}\label{inj.1}
d_{\pi(k_{n}-1)}\geq d_{\pi(k_{n})}\quad\text{where}\quad 
k_{n} := \min\left\{k\in \N:\sum_{i=1}^{k}d_{\pi(i)}\geq n\right\}.\end{equation}
\end{lem}

\begin{proof}For $n\in \N$ define
\begin{equation}\label{alg.1}
 m_{n}:= \min\left\{k\in \N:\sum_{i=1}^{k}d_{i}\geq n\right\}.
\end{equation}
Define a bijection $\pi_{n}:\{m_{n-1}+1,\ldots,m_{n}\}\to\{m_{n-1}+1,\ldots,m_{n}\}$ such that $\{d_{\pi(i)}\}_{i=m_{n-1}+1}^{m_{n}}$ is in nonincreasing order with the convention that $m_0=0$. Finally, define a bijection $\pi:\N \to \N$ by 
\[
\pi(i)=\pi_n(i) \qquad\text{if } m_{n-1}<i \le m_n,\  n\in \N.
\]

We claim that
\begin{equation}\label{alg.2}
m_{n-1}+2 \le k_n \le m_n \qquad\text{for all }n\in \N.
\end{equation}
Indeed, by the minimality of $m_{n-1}$ we have for $n\ge 2$,
\[
\sum_{i=1}^{m_{n-1}+1}d_{\pi(i)} = 
\sum_{i=1}^{m_{n-1}}d_{i}+d_{\pi(m_{n-1}+1)}<(n-1/2)+1/2=n.
\]
The above holds also holds trivially for $n=1$. Thus, $k_n >m_{n-1}+1$ for all $n\in \N$. On the other hand, we have
\[
\sum_{i=1}^{m_{n}}d_{\pi(i)} = 
\sum_{i=1}^{m_{n}}d_{i}\ge n.
\]
This yields $k_n \le m_n$ and, thus, \eqref{alg.2} is shown.
By \eqref{alg.2} we have $m_{n-1}+1 \le k_n-1 <k_n \le m_n$.
Since $\{d_{\pi(i)}\}_{i=m_{n-1}+1}^{m_n}$ is nonincreasing, this yields \eqref{inj.1}.
\end{proof}

\begin{thm}\label{cptalg} Let $\{d_{i}\}_{i\in I}$ be a sequence such that $d_{i_{0}}\in [0,1)$ for some $i_{0}\in I$, $d_{i}\in[0,\frac{1}{2}]$ for all $i\neq i_{0}$, and $\sum_{i\in I}d_{i}=\infty$. There exists an orthogonal projection $P$ with diagonal $\{d_{i}\}_{i\in I}$.
\end{thm}

\begin{proof} Since $I$ is a countable set and $\sum_{i\in I} d_{i} = \infty$ we may assume without loss of generality that $I = \N$ and $i_{0} = 1$. By Lemma \ref{inj} there is a bijection $\pi:\N\to\N$ such that \eqref{inj.1} holds.

For each $n\in \N$ set
\begin{equation}\label{alg.7}\sigma_{n} = n - \sum_{i=1}^{k_{n}-2}d_{\pi(i)}.\end{equation}
From the definition of $k_{n}$ we see that 
\begin{equation}\label{alg.8}\sigma_{n} = n - \sum_{i=1}^{k_{n}}d_{\pi(i)} + d_{\pi(k_{n}-1)} + d_{\pi(k_{n})}\leq d_{\pi(k_{n}-1)} + d_{\pi(k_{n})}.\end{equation}
From the minimality of $k_{n}$ and \eqref{inj.1} we see that
\[\sigma_{n} = n - \sum_{i=1}^{k_{n}-1}d_{\pi(i)} + d_{\pi(k_{n}-1)} \geq d_{\pi(k_{n}-1)}\geq d_{\pi(k_{n})},\]
which implies that
\begin{equation}\label{alg.9} \sigma_{n}\geq \max\{d_{\pi(k_{n}-1)}, d_{\pi(k_{n})}\}.\end{equation}
By Lemma \ref{translem} for each $n$ there exists $a_{n}\in[0,1]$ such that the matrix
\begin{equation*}\left[\begin{array}{cc}
a_{n} & \sigma_{n}-a_{n}\\
d_{\pi(k_{n}-1)}-a_{n} & d_{\pi(k_{n})}-\sigma_{n}+a_{n}
\end{array}\right]\end{equation*}
has non-negative entries and
\begin{equation}\label{orth}a_{n}(d_{\pi(k_{n}-1)}-a_{n})=(\sigma_{n}-a_{n})(d_{\pi(k_{n})}-\sigma_{n}+a_{n}).\end{equation}

Let $\{e_{i}\}_{i\in\N}$ be an orthonormal basis for a Hilbert space $\Hil$. Set
\[v_{1} = \sum_{i=1}^{k_{1}-2} d_{\pi(i)}^{1/2}e_{i} + a_{1}^{1/2}e_{k_{1}-1} - (\sigma_{1}-a_{1})^{1/2}e_{k_{1}},\]
and for $n\geq 2$ define
\begin{align*}
 v_{n} & = (d_{\pi(k_{n-1}-1)}-a_{n-1})^{1/2}e_{k_{n-1}-1} + (d_{\pi(k_{n-1})}-\sigma_{n-1}+a_{n-1})^{1/2}e_{k_{n-1}}\\
 & + \sum_{i=k_{n-1}+1}^{k_{n}-2}d_{\pi(i)}^{1/2}e_{i} + a_{n}^{1/2}e_{k_{n}-1} - (\sigma_{n}-a_{n})^{1/2}e_{k_{n}}.
\end{align*}
We can visualize $\{v_{n}\}_{n\in \N}$ as row vectors expanded in the orthonormal basis $\{e_i\}_{i\in I}$ by the following infinite matrix.
\[
\begin{bmatrix}
v_1 \\
\hline
v_2 \\
\hline
v_3
\\ \hline
\cdots
\end{bmatrix}
=
\left[\begin{array}{ccccccccccccc}
\sqrt{d_\cdot} & \cdots  & \sqrt{a_1} & -\sqrt{\sigma_1-a_1} 
\\
  &  & \sqrt{d_\cdot-a_1} & \sqrt{d_\cdot-\sigma_1+a_1} &\sqrt{d_\cdot} & \cdots  & \sqrt{a_2} & -\sqrt{\sigma_2-a_2} 
\\
& & & & & &\sqrt{d_\cdot-a_2} & \sqrt{d_\cdot-\sigma_2+a_2}&  \cdots  
\\
& & & & & & & & \cdots
\end{array}
\right]
\]
In the above matrix empty spaces represents $0$ and $d_\cdot$ is an abbreviation for $d_{\pi(i)}$ in $i$th column. 

We claim that $\{v_{n}\}_{n\in \N}$ is an orthonormal set in $\mathcal H$. Indeed, by \eqref{alg.7} we have for $n\geq 2$
\begin{align*}
\norm{v_{n}}^{2} & = d_{\pi(k_{n-1}-1)}-a_{n-1} + d_{\pi(k_{n-1})}-\sigma_{n-1}+a_{n-1} + \sum_{i=k_{n-1}+1}^{k_{n}-2}d_{\pi(i)} + a_{n} + \sigma_{n}-a_{n}\\
 & = \sum_{i=k_{n-1}-1}^{k_{n}-2}d_{\pi(i)} + \sigma_{n}-\sigma_{n-1}\\
 & = \sum_{i=k_{n-1}-1}^{k_{n}-2}d_{\pi(i)} + \left(n-\sum_{i=1}^{k_{n}-2}d_{\pi(i)}\right)-\left(n-1-\sum_{i=1}^{k_{n-1}-2}d_{\pi(i)}\right) = 1.
\end{align*}
A similar calculation yields $\norm{v_{1}}=1$. This means that rows of our infinite matrix have each norm $1$. Moreover, they are mutually orthogonal since any two vectors $v_n$ and $v_m$ have disjoint supports unless they are consecutive: $v_n$ and $v_{n+1}$. However, in the latter case the orthogonality is a consequence of \eqref{orth}.

Define the orthogonal projection $P$ by
\[Pv = \sum_{n\in\N} \langle v,v_{n}\rangle v_{n}, \qquad v\in \mathcal H.\]
It is easy to check that the $i$th column of our infinite matrix has norm equal to $\sqrt{d_{\pi(i)}}$. In other words, for each $i\in\N$ we have 
\[
\langle Pe_{i},e_{i}\rangle =||Pe_i||^2=\sum_{n\in \N} |\langle e_i,v_n\rangle|^2= d_{\pi(i)}.
\]
This completes the proof of Theorem \ref{cptalg}.
\end{proof}

We are now ready to prove Carpenter's Theorem under assumption (ii).

\begin{thm}\label{abinf} If $\{d_{i}\}_{i\in I}$ is a sequence in $[0,1]$ such that
\begin{equation}\label{abinf1}
a = \sum_{d_{i}<1/2}d_{i} = \infty\qquad\text{or}\qquad b =\sum_{d_{i}\geq 1/2}(1-d_{i}) = \infty,
\end{equation}
then there is a projection $P$ with diagonal $\{d_{i}\}$.
\end{thm}

\begin{proof} Set 
\[I_{0} = \{i : d_{i} \leq 1/2\}\quad \text{and}\quad I_{1} = \{i :d_{i} > 1/2\}.\]
Our hypothesis \eqref{abinf1} implies that
\begin{equation}\label{abinf2}
a' = \sum_{i\in I_{0}}d_{i} = \infty
\qquad\text{or}\qquad
b=\infty.
\end{equation}

{\bf{Case 1.}} Assume that $a'=\infty$. We can partition $I$ into countably many sets $\{J_{n}\}_{n\in \N}$ such that each $J_n$ contains at most one element in $I_1$ and
\[
\sum_{i\in J_{n}} d_{i} = \infty \qquad\text{for all }
n\in \N.\]
This is possible since $I_0$ satisfies \eqref{abinf2}.
By Theorem \ref{cptalg}, for each $n\in \N$ there is a projection $P_{n}$ with diagonal $\{d_{i}\}_{i\in J_n}$. Thus, the projection
\[P = \bigoplus_{n\in \N}P_{n}\]
has the desired diagonal $\{d_i\}_{i\in I}$. This completes the proof of Case 1.

{\bf{Case 2.}} Assume that $b=\infty$. Note that
\[b = \sum_{1-d_{i}\leq 1/2}(1-d_{i}).
\]
Thus, by Case 1 there is a projection $P'$ with diagonal $\{1-d_{i}\}$. Hence, $P=\mathbf I-P'$ is a projection with diagonal $\{d_{i}\}$.
\end{proof}

\section{A selector problem}

Kadison's Theorem \ref{Kadison} is closely connected with an open problem of characterizing all spectral functions of shift-invariant spaces. 
Shift-invariant (SI) spaces are closed subspaces of $L^2(\R^d)$ that are
invariant under all shifts, i.e., integer translations. That is, a closed subspace $V \subset  L^2(\R^d)$ is SI if $T_k(V)=V$ for all $k\in\Z^d$,
where $T_kf(x)=f(x-k)$ is the translation operator.
The theory of
shift-invariant spaces plays an important role in many
areas, most notably in the theory of wavelets,  spline systems, Gabor systems,
and approximation theory \cite{BDR1, BDR2, Bo, RS1, RS2}.
The study of analogous spaces for $L^2(\T, \mathcal H)$ with values in a
separable Hilbert space $\mathcal H$ in terms of the range function, often called
doubly-invariant spaces, is quite classical and goes back to Helson \cite{He}.

In the context of SI spaces a {\it range function} is any mapping
$$J: \T^d \to \{\text{closed subspaces of }\ell^2(\Z^d)\},$$
where $\T^d=\R^d/\Z^d$ is
identified with its fundamental domain $[-1/2,1/2)^d$. We say that $J$ is {\it
measurable} if the associated orthogonal projections $P_J(\xi): \ell^2(\Z^d) \to
J(\xi)$ are operator measurable, i.e., $\xi \mapsto P_J(\xi) v$ is measurable
for any $v\in \ell^2(\Z^d)$. We follow the convention which identifies range functions if they are equal a.e. 
A fundamental result due to Helson \cite[Theorem 8, p.~59]{He} gives one-to-one correspondence between SI spaces $V$ and measurable range functions $J$, see also \cite[Proposition 1.5]{Bo}.
Among several equivalent ways of introducing the spectral function of a SI space the most relevant definition uses a range function.

\begin{defn}
The {\it spectral function} of a SI space $V$ is a measurable mapping 
$\sigma_V: \R^d \to [0,1]$ given by
\begin{equation}\label{dsp}
\sigma_V(\xi+k) = ||P_J(\xi) e_k||^2=\langle P_J(\xi)e_k,e_k \rangle \qquad\text{for } \xi\in \T^d,\ k\in\Z^d,
\end{equation}
where $\{e_k\}_{k\in\Z^d}$ denotes the standard basis of $\ell^2(\Z^d)$ and $\T^d=[-1/2,1/2)^d$. In other words, $\{\sigma_V(\xi+k)\}_{k\in \Z^d}$ is a diagonal of a projection $P_J(\xi)$.
\end{defn}

Note that $\sigma_V(\xi)$ is well defined for a.e.~$\xi\in\R^d$,
since $\{ k+ \T^d: k\in \Z^d\}$
is a partition of $\R^d$. As an immediate consequence of Theorem \ref{Kadison} we have the following result.

\begin{thm}\label{sp}
Suppose that $V \subset L^2(\R^d)$ is a SI space. Let $\sigma=\sigma_V:\R^d \to [0,1]$ be its spectral function. For $\xi\in \T^d$ define
\[
a(\xi)=\sum_{k\in \Z^d, \ \sigma(\xi+k)<1/2} \sigma(\xi+k) \quad\text{and}\quad 
b(\xi)=\sum_{k\in \Z^d, \ \sigma(\xi+k) \ge1/2}(1-\sigma(\xi+k)).
\]
Then, for a.e. $\xi \in \R^d$ we either have
\begin{enumerate}
\item $a(\xi),b(\xi)<\infty$ and  $a(\xi)-b(\xi)\in\Z$, or
\item $a(\xi)=\infty$ or $b(\xi)=\infty$.
\end{enumerate}
\end{thm}

It is an open problem whether the converse to Theorem \ref{sp} holds. 

\begin{problem}
Suppose that a measurable function $\sigma:\R^d \to [0,1]$ satisfies either (i) or (ii) for a.e. $\xi \in \R^d$. Does there exists a SI space $V\subset L^2(\R^d)$ such that its spectral function $\sigma_V=\sigma$?
\end{problem}

The sufficiency part of Theorem \ref{Kadison}, i.e., Carpenter's Theorem, suggests a positive answer to this problem. That is, for a.e. $\xi$ it yields a projection $P_J(\xi)$ whose diagonal satisfies \eqref{dsp}. However, it does not guarantee a priori that the corresponding range function $J$ is measurable. This naturally leads to the following selector problem.

\begin{problem}
Let $X$ be a finite (or $\sigma$-finite) measure space and let $I$ be a countable index set. Let $\sigma:X \times I \to [0,1]$ be a measurable function. For $\xi\in X$ define
\[
a(\xi)=\sum_{i\in I, \ \sigma(\xi,i)<1/2} \sigma(\xi,i) \quad\text{and}\quad 
b(\xi)=\sum_{i\in I, \ \sigma(\xi,i)\ge 1/2} (1-\sigma(\xi,i)).
\]
Suppose that for a.e. $\xi \in X$ we either have
\begin{enumerate}
\item $a(\xi),b(\xi)<\infty$ and  $a(\xi)-b(\xi)\in\Z$, or
\item $a(\xi)=\infty$ or $b(\xi)=\infty$.
\end{enumerate}
Does there exists a measurable range function 
$J:X \to \{\text{closed subspaces of }\ell^2(I)\}$ such that the corresponding orthogonal projections $P_J(\xi)$ have diagonal $\{\sigma(\xi,i)\}_{i\in I}$ for a.e. $\xi \in X$?
\end{problem}

In other words, Problem 2 asks whether it is possible to find a measurable selector of projections in Theorem \ref{Kadison}. The constructive proof of Carpenter's Theorem given in this paper might be a first step toward resolving this problem. However, Problem 2 remains open.


\begin{thebibliography}{99}

\bibitem{ar}
M. Argerami,
{\it Majorisation and Kadison's Carpenter's Theorem},
preprint {\tt arXiv:1304.1232}.

\bibitem{am}
M. Argerami, P. Massey,
{\it Towards the Carpenter's theorem}, Proc. Amer. Math. Soc. {\bf 137} (2009), 3679--3687.

\bibitem{a}
W. Arveson, 
{\it Diagonals of normal operators with finite spectrum},
Proc. Natl. Acad. Sci. USA {\bf 104} (2007), 1152--1158.

\bibitem{ak}
W. Arveson, R. Kadison, 
{\it Diagonals of self-adjoint operators}, Operator theory, operator algebras, and applications, 247--263, Contemp. Math., {\bf 414}, Amer. Math. Soc., Providence, RI, 2006. 

\bibitem{BDR1}
C. de Boor, R. A. DeVore, A. Ron,
{\it The structure of finitely generated shift-invariant spaces in
${L}\sb 2({\R}^d)$}, 
J. Funct. Anal. {\bf 119} (1994), 37--78.

\bibitem{BDR2}
C. de Boor, R. A. DeVore, and A. Ron,
{\it Approximation orders of {FSI} spaces in {$L\sb 2(\R\sp d)$}},
Constr. Approx.
{\bf 14} (1998), 631--652.


\bibitem{Bo}
M. Bownik, 
{\it The structure of shift-invariant subspaces of $L\sp 2(\mathbb R\sp n)$}, J. Funct. Anal. {\bf 177} (2000), 282--309. 


\bibitem{mbjj}
M. Bownik, J. Jasper,
{\it Characterization of sequences of frame norms},
J. Reine Angew. Math. {\bf 654} (2011), 219--244.

\bibitem{br}
M. Bownik, Z. Rzeszotnik, 
{\it The spectral function of shift-invariant spaces},
Michigan Math. J. {\bf 51} (2003), 387--414. 

\bibitem{cfmwz}
P. Casazza, M. Fickus, D. Mixon, Y. Wang, Z. Zhou, 
{\it Constructing tight fusion frames}, 
Appl. Comput. Harmon. Anal. {\bf 30} (2011), 175--187. 

\bibitem{chkwz}
P. Casazza, A. Heinecke, K. Kornelson, Y. Wang, Z. Zhou,
{\it Necessary and sufficient conditions to perform Spectral Tetris}, Linear Algebra Appl.  (to appear).

\bibitem{He}
H. Helson, {\it Lectures on invariant subspaces},
Academic Press, New York-London, 1964.

\bibitem{horn}
A. Horn,
{\it  Doubly stochastic matrices and the diagonal of a rotation matrix},
Amer. J. Math. {\bf 76} (1954), 620--630.

\bibitem{k1}
R. Kadison, 
{\it The Pythagorean theorem. I. The finite case},
Proc. Natl. Acad. Sci. USA {\bf 99} (2002), 4178--4184.

\bibitem{k2}
R. Kadison, 
{\it The Pythagorean theorem. II. The infinite discrete case},
Proc. Natl. Acad. Sci. USA {\bf 99} (2002), 5217--5222.


\bibitem{kw0}
V. Kaftal, G. Weiss,
{\it A survey on the interplay between arithmetic mean ideals, traces, lattices of operator ideals, and an infinite Schur-Horn majorization theorem}, Hot topics in operator theory, 101--135, Theta Ser. Adv. Math., 9, Theta, Bucharest, 2008.

\bibitem{kw}
V. Kaftal, G. Weiss,
{\it An infinite dimensional Schur-Horn theorem and majorization theory},  J. Funct. Anal. {\bf 259} (2010), 3115--3162.

\bibitem{moa}
A. W. Marshall, I. Olkin, B. C. Arnold, {\it Inequalities: theory of majorization and its applications}. Second edition. Springer Series in Statistics. Springer, New York, 2011.

\bibitem{RS1}
A. Ron, Z. Shen, 
{\it Affine systems in ${L}\sb 2({\mathbb R}\sp d)$:
the analysis of the analysis operator}, J. Funct. Anal.
{\bf 148} (1997), 408--447.

\bibitem{RS2}
A. Ron, Z. Shen, 
{\it Weyl-Heisenberg frames and Riesz bases in ${L}\sb 2({\mathbb R}\sp d)$},
Duke Math. J. {\bf 89} (1997), 237--282.

\bibitem{schur}
I. Schur,
{\it \"Uber eine Klasse von Mittelbildungen mit Anwendungen auf die Determinantentheorie}, 
Sitzungsber. Berl. Math. Ges. {\bf 22} (1923), 9--20.




\end{thebibliography}
\end{document}